\documentclass[12pt]{amsart}

\usepackage{fullpage}

\usepackage{amsfonts,amssymb,amscd,amsmath,enumerate,verbatim,color}
\usepackage[latin1]{inputenc}
\usepackage{amscd}
\usepackage{latexsym}
\usepackage{tikz}
\usepackage{mathptmx}
\usepackage{mathrsfs}
\usepackage{amsthm}
\usetikzlibrary{intersections, calc, arrows}

\newtheorem{thm}{Theorem} [section]
 
\newtheorem{prop}[thm]{Proposition} 
\newtheorem{lemma}[thm]{Lemma}
\newtheorem{cor}[thm]{Corollary}

\theoremstyle{definition}
\newtheorem{definition}[thm]{Definition}
\newtheorem{example}[thm]{Example}
\newtheorem{conjecture}[thm]{Conjecture}
\newtheorem{Remark}[thm]{Remark}
\numberwithin{thm}{section}  
\numberwithin{equation}{section} 
\usetikzlibrary{intersections, calc, arrows}

\def\epoly{{\mathcal P}_G}

\def\NZQ{\mathbb}               

\def\QQ{{\NZQ Q}}
\def\ZZ{{\NZQ Z}}
\def\RR{{\NZQ R}}

\def\ab{{\mathbf a}}

\def\eb{{\mathbf e}}
\def\tb{{\mathbf t}}

\def\wb{{\mathbf w}}

\def\a{\alpha}

\def\opn#1#2{\def#1{\operatorname{#2}}} 
\opn\gr{gr}

\def\int{{\rm int}}

\title{Simplex faces and quadratic toric ideals of lattice polytopes}
\author{Aki Mori and Hidefumi Ohsugi}
\date{}
\address{Aki Mori,
	Center for Physics and Mathematics, 
    Institute for Liberal Arts and Sciences, 
    Osaka Electro-Communication University, 
    Neyagawa, Osaka 572-8530, Japan} 
\email{a-mori@osakac.ac.jp}

\address{Hidefumi Ohsugi,
	Department of Mathematical Sciences,
	School of Science,
	Kwansei Gakuin University,
	Sanda, Hyogo 669-1330, Japan} 
\email{ohsugi@kwansei.ac.jp}


\begin{document}

\begin{abstract}
We say that a convex polytope has the clique-face property if every clique in its 1-skeleton is the vertex set of a face.
We establish this property as a geometric necessary condition for quadratic generation of toric ideals. More precisely, we prove that every lattice polytope with primitive edges and a quadratic toric ideal has the clique-face property; in particular, this holds for every 
$(0,1)$-polytope with a quadratic toric ideal. 
For $(0,1)$-polytopes satisfying condition (E), we characterize the clique-face property in terms of divisibility by monomials occurring in quadratic binomials, 
and show that, under the clique-face property, such toric ideals have no indispensable monomials of degree at least three.
For edge polytopes and cut polytopes, we prove that the clique-face property is equivalent to quadratic generation. 
This yields new geometric characterizations of quadratic generation for these classes.
We also prove that all simple polytopes, matroid independence polytopes, and matroid base polytopes have the clique-face property, and discuss the case of stable set polytopes in connection with conjectures on quadratic toric ideals.
\end{abstract}

\maketitle

\section{Introduction}
Let $P \subset \RR^d$ be a convex polytope
and let $V(P)$ be the set of all vertices of $P$. 
The \textit{$1$-skeleton} ${\rm sk}(P)$ of $P$ is the finite simple graph on the vertex set $V(P)$ whose edge set is
$$
\{\{v, w\} : v,w \in V(P), {\rm conv}(\{v,w\}) \mbox{ is an edge of } P\}.
$$
A \textit{clique} of a graph $G$ is the vertex set of a complete subgraph of $G$.   
We say that a convex polytope \(P\) has the \textit{clique-face property} if

\smallskip

\begin{center}
    every clique of ${\rm sk}(P)$ is the vertex set of a face of $P$.   
\end{center}

\smallskip

\noindent
We denote this property by \((*)\).
Here, we regard $P$ itself as a face of $P$,
and hence simplices satisfy the clique-face property ($*$).
We also note that if \(P\) has the clique-face property, then so does every face of \(P\). Consequently, every clique of \({\rm sk}(P)\) is the vertex set of a simplex face of \(P\). Indeed, if \(C\) is the vertex set of a face \(F\), then every subset of \(C\) is again a clique in \({\rm sk}(F)\), and hence is the vertex set of a face of \(F\). Therefore, \(F\) is a simplex.

For example, a simplicial polytope $P$ with ($*$) is called {\it flag}.
Any polytope $P$ of dimension $\le 2$ satisfies ($*$).
If \(\dim P=3\), then by Whitney's theorem \cite{W}, the boundary cycles of the
2-faces of \(P\) are precisely the induced non-separating cycles of
\({\rm sk}(P)\). 
%
Thus a 3-polytope $P$ satisfies ($*$)
if and only if
there is no triangle in ${\rm sk}(P)$ that separates ${\rm sk}(P)$.
It is known \cite{Mor} that order polytopes and chain polytopes of posets satisfy ($*$).

The purpose of this paper is to introduce the clique-face property and establish it as a geometric necessary condition for quadratic generation of toric ideals of lattice polytopes. We also identify important classes of \((0,1)\)-polytopes for which this geometric condition characterizes quadratic generation.
Here, a \textit{lattice polytope} is a convex polytope
all of whose vertices are integer vectors.
Let $P \subset \RR^d$ be a lattice polytope with $P \cap \ZZ^d = \{\ab_1,\ldots,\ab_n\}$.
Let $K[x_1,\dots, x_n]$ be a polynomial ring in $n$ variables over a field $K$
and let $K[\tb^{\pm 1},s]:= K[t_1,t_1^{-1} ,\dots, t_d, t_d^{-1}, s ]$ be a Laurent polynomial ring
in $d+1$ variables over $K$.
Given an integer vector $\ab=(a_1,\dots,a_d) \in \ZZ^d$, 
we set $\tb^\ab=t_1^{a_1} \dots t_d^{a_d} \in K[\tb^{\pm 1},s]$.
Then the \textit{toric ideal} $I_P$ of $P$ is the kernel of 
a ring homomorphism $\pi : K[x_1,\dots, x_n] \rightarrow K[\tb^{\pm 1},s]$ defined by $\pi(x_i) = {\bf t}^{\ab_i} s$ for each $i= 1,2,\dots,n$.
It is known that $I_P$ is generated by homogeneous binomials.
In addition, we have
\begin{equation} \label{basicIP}
    x_1^{\alpha_1} \cdots x_n^{\alpha_n} 
-
x_1^{\beta_1} \cdots x_n^{\beta_n} \in I_P
\iff
\alpha_1 \ab_1 + \cdots + \alpha_n \ab_n
=
\beta_1 \ab_1 + \cdots + \beta_n \ab_n.
\end{equation}
See, e.g., \cite{HHO} for details on toric ideals.
The toric ideal $I_P$ is said to be {\it quadratic}
if $I_P$ is generated by quadratic binomials.
A lattice segment is called \textit{primitive} if its only lattice points are its endpoints.
Our first main result shows that the clique-face property is a geometric necessary condition for quadratic generation, provided that all edges of the polytope are primitive.
\begin{thm}
\label{quadGENE}
Let $P \subset \RR^d$ be a lattice polytope
such that every edge of $P$ is primitive.
If $I_P$ is quadratic,
    then any clique of ${\rm sk}(P)$ is the vertex set of a face of $P$.
\end{thm}

If $P \subset \RR^d$ is a $(0,1)$-polytope, then 
$P \cap \ZZ^d$ is the vertex set of $P$.
In particular, every edge of $P$ is primitive.
Hence we have the following from Theorem \ref{quadGENE}.

\begin{cor}\label{zero_one}
    Let $P \subset \RR^d$ be a $(0,1)$-polytope.
If $I_P$ is quadratic,
    then any clique of ${\rm sk}(P)$ is the vertex set of a face of $P$.
\end{cor}

For example, it is known that the toric ideals of $(0,1)$-polytopes belonging to any of the following classes are quadratic:
\begin{itemize}
\item 
hypersimplices;
    \item 
    order polytopes of posets;
    \item
    chain polytopes of posets;
    \item 
    stable set polytopes of weakly chordal graphs, Meyniel graphs,
    perfectly orderable graphs, and almost bipartite graphs;
    \item 
    edge polytopes of chordal bipartite graphs and complete multipartite graphs;
    \item 
    cut polytopes of graphs without $K_4$-minors.
\end{itemize}
From Corollary \ref{zero_one}, $P$ has the clique-face property ($*$) if $P$ belongs to one of the above classes.

On the other hand, the converse of Theorem \ref{quadGENE} does not hold in general (e.g., Example \ref{counterexam}).
However, under condition (E), the clique-face property ($*$) guarantees a necessary condition for quadratic toric ideals.
In \cite{propertyE}, a $(0,1)$-polytope $P$ is said to satisfy (E) if 

\medskip

\begin{enumerate}
    \item[(E)] 
 two distinct vertices $\ab_i$ and $\ab_j$ of $P$ form an edge of $P$ if and only if
 there exists a unique way to write $\ab_i+\ab_j$ as the sum of two vertices of $P$.
\end{enumerate}
    
\medskip

\noindent
Note that the ``only if" part holds in general.
 It was shown in \cite{propertyE} that any matroid polytope
 and any stable set polytope of a graph satisfy (E).
A binomial $f$ in $I_P$ is said to be {\it indispensable}
if either $f$ or $-f$ appears in every system of binomial generators of $I_P$.
 A monomial $m$ is said to be {\it indispensable} if every system of binomial generators of $I_P$ 
 contains a binomial $f$ such that $m$ is a monomial in $f$.
 In particular, two monomials in an indispensable binomial are indispensable.

\begin{thm}\label{indispensable}

Let \(P\) be a \((0,1)\)-polytope satisfying condition (E).
Then \(P\) has the clique-face property \((*)\) if and only if
any monomial appearing in a binomial in \(I_P\) is divisible by
a quadratic monomial appearing in a quadratic binomial in \(I_P\).
In particular, if \(P\) has the clique-face property \((*)\), then
\(I_P\) has no indispensable monomials of degree \(\ge 3\).
\end{thm}

From Corollary \ref{zero_one} and Theorem \ref{indispensable}, for a $(0,1)$-polytope $P$ with (E), we have 
\begin{equation} \label{three_conditions}
    I_P \mbox{ is quadratic } \Rightarrow
(*)
\Rightarrow
I_P \mbox{ has no indispensable monomials of degree } \ge 3.  
\end{equation}
In the present paper,
for the following classes of $(0,1)$-polytopes,
we will show that 
these polytopes satisfy (E), and that
three conditions in \eqref{three_conditions} are equivalent:
\begin{itemize}
    \item 
    Edge polytopes of graphs;

    \item 
    Cut polytopes of graphs.
   
\end{itemize}
Next, we study classes of lattice polytopes whose toric ideals are conjectured to be quadratic.
Two important classes of lattice polytopes whose toric ideals are conjectured to be quadratic are smooth polytopes and matroid base polytopes.
Note that any smooth polytope $P$ is simple, i.e., 
${\rm sk}(P)$ is a $d$-regular graph with $d = \dim P$.
We will show that
\begin{itemize}
    \item 
    Any simple polytope satisfies ($*$);

    \item 
    Any matroid independence polytope satisfies ($*$);

        \item 
    Any matroid base polytope satisfies ($*$).
\end{itemize}

The present paper is organized as follows:
In Section 2, we give a sufficient condition for ($*$) in terms of toric ideals (Theorem \ref{more general fact}).
From this, it follows that quadratic generation guarantees the clique-face property when every edge of the polytope is primitive (Theorem \ref{quadGENE}).
In particular, for $(0,1)$-polytopes, quadratic generation guarantees ($*$) (Corollary \ref{zero_one}).
In Section 3, we study $(0,1)$-polytopes with (E).
        A $(0,1)$-polytope $P$ with (E) satisfies ($*$) if and only if any monomial appearing in a binomial in $I_P$ is divisible by a quadratic monomial appearing in a quadratic binomial in $I_P$ (Theorem \ref{indispensable}).
In particular,     
    if a $(0,1)$-polytope $P$ satisfies (E) and ($*$), then $I_P$ has no indispensable monomials of degree $\ge 3$.  
In addition, we will show that (i) edge polytopes of graphs and (ii) cut polytopes of graphs satisfy (E) and that
three conditions in \eqref{three_conditions} are equivalent for these polytopes.
In Section 4, we study classes of lattice polytopes whose toric ideals are conjectured to be quadratic.
We will show that any simple polytope, any matroid independence polytope, and any matroid base polytope satisfy ($*$).
Finally, we discuss stable set polytopes of graphs in connection with a conjecture.

\subsection*{Acknowledgment}
This work was supported by 
JSPS KAKENHI 24K00534 and 26K17026.

\section{Quadratic toric ideals and the clique-face property}

In the present section, we prove Theorem \ref{quadGENE}.
The following fact is fundamental.

\begin{lemma}\label{key}
Let $P \subset \RR^d$ be a lattice polytope
such that every edge of $P$ is primitive.
Let $P \cap \ZZ^d = \{\ab_1,\ldots,\ab_n\}$.
If a nonzero binomial $x_i x_j - x_k x_\ell$ 
belongs to $I_P$ with $i\neq j$, then ${\rm conv}(\ab_i , \ab_j)$ is not an edge of $P$.
\end{lemma}

\begin{proof}
From (\ref{basicIP}), we have $\ab_i+\ab_j=\ab_k+\ab_\ell$.
Suppose, to the contrary, that 
${\rm conv}(\ab_i,\ab_j)$ is an edge of $P$. 
Since ${\rm conv}(\ab_i,\ab_j)$ is a face of $P$, there exist $\wb\in\RR^d$ and $b\in\RR$ such that \[ \wb\cdot\ab\leq b \quad\text{for all $\ab\in P$}, \] and \[ {\rm conv}(\ab_i,\ab_j)=\{\ab\in P:\wb\cdot\ab=b\}. \] 
In particular, \[ \wb\cdot\ab_i=\wb\cdot\ab_j=b. \] 
Taking the inner product of $\ab_i+\ab_j=\ab_k+\ab_\ell$ with $\wb$, we obtain \[ 2b = \wb\cdot(\ab_i+\ab_j) = \wb\cdot(\ab_k+\ab_\ell) \leq 2b. \] 
Since both $\wb\cdot\ab_k$ and $\wb\cdot\ab_\ell$ are at most $b$, equality implies \[ \wb\cdot\ab_k=\wb\cdot\ab_\ell=b. \] Hence $\ab_k,\ab_\ell\in {\rm conv}(\ab_i,\ab_j)$. 
Since ${\rm conv}(\ab_i,\ab_j)$ is primitive, its only lattice points are its endpoints. 
Thus ${\rm conv}(\ab_i,\ab_j)\cap\ZZ^d=\{\ab_i,\ab_j\}$, and hence $\ab_k,\ab_\ell\in\{\ab_i,\ab_j\}$. 
Since $i\neq j$ and $\ab_k+\ab_\ell=\ab_i+\ab_j$, it follows that $\{\ab_k,\ab_\ell\}=\{\ab_i,\ab_j\}$ as multisets. 
Therefore, $x_kx_\ell=x_ix_j$, which contradicts the assumption that $x_ix_j-x_kx_\ell$ is nonzero. 
Thus ${\rm conv}(\ab_i,\ab_j)$ is not an edge of $P$. 
\end{proof}

Note that ``every edge of $P$ is primitive" is needed in Lemma \ref{key}.
For example, although
\[
(0,1) + (3,1) = (1,1) + (2,1) 
\]
holds,
${\rm conv}(\{ (0,1), (3,1) \})$ might be an edge of $P$.
From the theory of combinatorial pure subrings \cite{cpureR1, cpure}, we have the following.

\begin{lemma}\label{cp}
    Let $Q$ be a face of a lattice polytope $P$.
    Let $\{\ab_1,\dots,\ab_\ell\}$ be the vertex set of $Q$
    and let $\{\ab_1,\dots,\ab_m\}$ be the vertex set of $P$.    
    Then we have the following{\rm :}
    \begin{itemize}
        \item[(i)]    If $I_P$ is quadratic,
    then so is $I_Q$.
    \item[(ii)] 
    If every binomial in \(I_P \cap K[x_1,\dots,x_m]\) is generated by quadratic binomials in \(I_P\), then every binomial in 
\(I_Q \cap K[x_1,\dots,x_\ell]\) is generated by quadratic binomials in \(I_Q\).  
\item[(iii)] Every indispensable monomial of $I_Q$ is an indispensable monomial of $I_P$. 
\item[(iv)] Every indispensable binomial of $I_Q$ is an indispensable binomial of $I_P$.
    \end{itemize}
 
\end{lemma}

\begin{proof}
    Part (i) follows from the theory of combinatorial pure subrings \cite[Theorem 1.3]{cpureR1}.
Let $\{\ab_1,\dots, \ab_n\}$ be the set of lattice points in $P$
and let
\[
\rho:K[x_1,\ldots,x_n]\longrightarrow R:=K[x_k : \ab_k\in Q]
\]
be the \(K\)-algebra homomorphism defined by
\[
\rho(x_k)=
\begin{cases}
x_k & \text{if } \ab_k\in Q,\\
0 & \text{if } \ab_k\notin Q.
\end{cases}
\]

For (ii), let \(f\in I_Q\cap K[x_1,\ldots,x_\ell]\). 
Then
\(f\in I_P\cap K[x_1,\ldots,x_m]\), and hence, by the assumption,
\begin{equation} \label{gene equ}
    f=\sum_{\mu=1}^s m_\mu f_\mu,
\end{equation}
where each \(f_\mu=x_{\alpha_\mu} x_{\beta_\mu} - x_{\gamma_\mu} x_{\delta_\mu} \)
is a quadratic binomial in \(I_P\).
Since \(Q\) is a face of \(P\), for each  \(f_\mu=x_{\alpha_\mu} x_{\beta_\mu} - x_{\gamma_\mu} x_{\delta_\mu}\in I_P\),
we have
\[
x_{\alpha_\mu} x_{\beta_\mu}\in R \quad \Longleftrightarrow \quad   x_{\gamma_\mu} x_{\delta_\mu}\in R.
\]
In fact, this follows by taking the inner product of $\ab_{\alpha_\mu} +\ab_{\beta_\mu} =\ab_{\gamma_\mu} +\ab_{\delta_\mu}$ with a supporting vector of the face \(Q\).
Thus \(\rho(f_\mu)=f_\mu\) if \(f_\mu\in I_Q\), and  \(\rho(f_\mu)=0\)
if \(f_\mu\notin I_Q\).
Since \(f\in R\), applying \(\rho\) to  \eqref{gene equ} gives
\[
f=\rho(f)
=\sum_{\mu=1}^s \rho(m_\mu)\rho(f_\mu)
=\sum_{f_\mu \in I_Q} \rho(m_\mu) f_\mu.
\]
Hence \(f\) is generated by quadratic binomials in \(I_Q\).

For (iii) and (iv), let $\mathcal{G}$ be an arbitrary system of binomial generators of $I_P$. By the same argument as in the proof of (ii), the set $\mathcal{G}_Q = \mathcal{G} \cap R$ is a system of binomial generators of $I_Q$.

Let $u$ be an indispensable monomial of $I_Q$. Since $\mathcal{G}_Q$ is a system of binomial generators of $I_Q$, there exists a binomial in $\mathcal{G}_Q$ having $u$ as one of its monomials. Hence $u$ appears in a binomial belonging to $\mathcal{G}$. Since $\mathcal{G}$ was arbitrary, $u$ is an indispensable monomial of $I_P$. This proves (iii). 

Let $f$ be an indispensable binomial of $I_Q$. Since $\mathcal{G}_Q$ is a system of binomial generators of $I_Q$, either $f$ or $-f$ belongs to $\mathcal{G}_Q$, and hence to $\mathcal{G}$. Since $\mathcal{G}$ was arbitrary, $f$ is an indispensable binomial of $I_P$. This proves (iv).
\end{proof}

Theorem \ref{quadGENE} follows from the more general fact stated below.

\begin{thm}\label{more general fact}
Suppose that  $P \subset \RR^d$ is a lattice polytope 
such that every edge of $P$ is primitive.
Let $\{\ab_1,\ldots,\ab_m\}$ be the vertex set of $P$ and let
$P \cap \ZZ^d = \{\ab_1,\ldots,\ab_n\}$ with $m \le n$.
If every binomial in $I_P \cap K[x_1,\dots,x_m]$ is generated by quadratic binomials in $I_P$, then any clique of ${\rm sk}(P)$ is the vertex set of a face of $P$.
\end{thm}

\begin{proof}
Let $\{\ab_{i_1} ,\dots,  \ab_{i_r}\} \subset \{\ab_1 ,\dots,  \ab_m\}$ be a clique 
of ${\rm sk}(P)$.
Suppose that 
$P'= {\rm conv}(\ab_{i_1} ,\dots,  \ab_{i_r})$ 
is not a face of $P$.
If $r=m$,
then 
$P=P'$,
a contradiction.
Hence $r<m$.

The proof is by induction on the dimension of $P$.
If $\dim P \le 2$, then $P$ satisfies the clique-face property ($*$).
Let $d=\dim P \ge 3$ and assume that the assertion holds for lattice polytopes of dimension $\le d-1$.
If $P'$ is contained in a facet $F$ of $P$,
then we can replace $P$ with $F$ by Lemma~\ref{cp}.
Thus we may assume that $P'$ is not contained in any facet of $P$.
Then $P'$ contains an inner point $v \in \QQ^d$ of $P$.
It then follows that
\[
v = \sum_{k=1}^r \lambda_k \ab_{i_k}
=\sum_{\ell=1}^m \mu_\ell \ab_\ell,
\]
for some $0 \leq \lambda_k \in \QQ$, $0 < \mu_\ell \in \QQ$,
and $\sum_{k=1}^r \lambda_k = \sum_{\ell=1}^m \mu_\ell = 1$.
Multiplying a suitable integer,
we have
\[
\sum_{k=1}^r \lambda_k' \ab_{i_k}
=\sum_{\ell=1}^m \mu_\ell' \ab_\ell,
\]
where 
$0 \leq \lambda_k' \in \ZZ$, $0 < \mu_\ell' \in \ZZ$,
and $\sum_{k=1}^r \lambda_k' = \sum_{\ell=1}^m \mu_\ell'$.
From (\ref{basicIP}),
    $$
    f:=
\prod_{k=1}^r x_{i_k}^{ \lambda_k'}
-
\prod_{\ell=1}^m x_\ell^{\mu_\ell'}  
$$
belongs to $I_P \cap K[x_1,\dots,x_m]$.
Since $r<m$ and $0 < \mu_\ell' \in \ZZ$ for $\ell = 1,2 \dots, m$, we have $f \ne 0$.
In addition, since the binomial $f$ is generated by quadratic binomials in $I_P$, there exists a binomial
$g:=x_i x_j - x_k x_\ell \in I_P$ 
such that $x_i x_j$ divides $\prod_{k=1}^r x_{i_k}^{ \lambda_k'}$.
Then $i,j \in \{i_1,\dots,i_r\}$
and $\ab_i + \ab_j = \ab_k + \ab_\ell$.
If $i = j$, then  $\ab_i  = \frac{1}{2}( \ab_k + \ab_\ell)$.
Since $\ab_i$ is a vertex of $P$, we have $i =k = \ell$.
Then $g = 0$, a contradiction.
Hence $i \ne j$.
From Lemma \ref{key},  ${\rm conv}(\ab_i , \ab_j)$ is not an edge of $P$.
This contradicts that
$\{\ab_{i_1} ,\dots,  \ab_{i_r}\}$ is a clique of ${\rm sk}(P)$.
\end{proof}

\section{$(0,1)$-polytopes satisfying (E)}

The converse of Theorem \ref{quadGENE} does not hold even in the two-dimensional case.

\begin{example}\label{counterexam}
Let 
    $P={\rm conv} (\{(1,0,1),(0,1,1),(-1,-1,1), (0,0,1)\})$.
Then $P$ is a triangle such that every edge is primitive.
On the other hand, $I_P$ is generated by $x_1 x_2 x_3 - x_4^3$.
    Thus $P$ is a counterexample of the converse of Theorem \ref{quadGENE}.

\end{example}

However, the following fact shows that this example is exceptional.

\begin{prop}[{\cite[Corollary 3.2.5]{BGT}}]
Let $P \subset \RR^2$ be a lattice polytope with $|P \cap \ZZ^2| \ge 4$.
Then the following conditions are equivalent:
\begin{itemize}
    \item[{\rm (i)}] $|\partial P \cap \ZZ^2| \ge 4${\rm ;}
    \item[{\rm (ii)}] $I_P$ has a quadratic Gr\"{o}bner basis{\rm ;}
    \item[{\rm (iii)}] the toric ring of $P$ is Koszul{\rm ;}
    \item[{\rm (iv)}] $I_P$ is quadratic.
\end{itemize}
    
\end{prop}

We now prove Theorem \ref{indispensable} that states that the clique-face property ($*$) and property (E) guarantee 
a necessary condition for quadratic generation of toric ideals.
    
\begin{proof}[Proof of Theorem \ref{indispensable}]
``If" part follows from the argument in the proof of Theorem \ref{more general fact}.

Let $P$ be a $(0,1)$-polytope with properties (E) and ($*$)
and let $P \cap \ZZ^d = \{\ab_1,\dots,\ab_n\}$.
Since $P$ is a $(0,1)$-polytope, $P \cap \ZZ^d$ coincides with the set of vertices of $P$.
    Suppose that $m=x_{i_1} \dots x_{i_r}$ ($r \ge 3$) appears in a 
    binomial $f=m -x_{j_1} \dots x_{j_r}$
    $(\ne 0)$ in $I_P$ and that 
    $m$ is not divisible by any monomial appearing in any quadratic binomial in $I_P$.
    Let $C_1:=\{\ab_{i_1}, \dots, \ab_{i_r}\}$ and $C_2:=\{\ab_{j_1}, \dots, \ab_{j_r}\}$.
Since $f$ is nonzero, $C_1$ and $C_2$
    are different as multisets.   
Suppose that ${\rm conv}(\{\ab_\alpha, \ab_\beta\})$ with $\ab_\alpha, \ab_\beta \in C_1$
and $\ab_\alpha \ne \ab_\beta$
is not an edge of $P$.
Since $P$ satisfies (E),
$\ab_\alpha+\ab_\beta = \ab_p + \ab_q$
for some $\{\alpha, \beta\} \ne \{p,q\}$.
Then 
the squarefree quadratic binomial
$x_\alpha x_\beta - x_p x_q$ belongs to $I_P$.
This contradicts the hypothesis for $m$.
Thus $C_1$ is a clique of ${\rm sk}(P)$.
Hence it is enough to show that  
${\rm conv}(C_1)$ is not a face of $P$.
Suppose that ${\rm conv}(C_1)$ is a face of $P$.
By the observation in the Introduction, it is a simplex.
    
    Since $f$ belongs to $I_P$, we have 
    \[\sum_{k=1}^r \ab_{i_k} = \sum_{k=1}^r \ab_{j_k},\]
    and hence 
\begin{equation}\label{jushin}
\sum_{k=1}^r \frac{1}{r} \ab_{i_k} = \sum_{k=1}^r \frac{1}{r} \ab_{j_k}.
\end{equation}
Since ${\rm conv}(C_1)$ is a face of $P$,
there exists a vector $\wb \in \RR^d$ and $b \in \RR$ such that
\[
\wb \cdot \ab_i 
\begin{cases}
< b &  \mbox{if } \ab_i \in (P\cap \ZZ^d ) \setminus C_1,\\
= b & \mbox{if } \ab_i \in C_1.
\end{cases}
\]
Taking the inner product $\wb$ with (\ref{jushin}),
we have 
$\sum_{k=1}^r \frac{1}{r} \wb \cdot \ab_{j_k} = b$.
It then follows that $C_2 \subset C_1$. 
Then $f$ is a binomial in $I_{{\rm conv}(C_1)}$.
However, since ${\rm conv}(C_1)$ is a simplex, we have 
$I_{{\rm conv}(C_1)}= \{0\}$, a contradiction.
\end{proof}

\subsection{Edge polytopes}

Let $G$ be a simple graph on the vertex set $[d]:= \{1,2,\dots,d\}$
with the edge set $E(G)=\{e_1,\dots, e_n \}$.
Then the {\em edge polytope} $\epoly$ of $G$ is the convex hull of
\[
\{
\eb_i + \eb_j : \{i,j\} \in E(G)
\},
\]
where $\eb_i \in \RR^d$ is a unit coordinate vector.
Numerous studies on edge polytopes have explored in detail
 from both algebraic and discrete geometric perspectives.
See the textbooks \cite{HHO, V} and the references therein.
In particular, a graph-theoretical characterization of systems of binomial generators of the toric ideal $I_{\epoly}$ is known \cite{Vecw}.
A {\it walk} of length $q$ of $G$ from $v_1 \in V(G)$ to $v_{q+1} \in V(G)$ is a sequence of the form 
\begin{equation} \label{ecw}
    \Gamma =(\{v_1,v_2\}, \{v_2,v_3\}, \dots,\{v_q,v_{q+1}\})
\end{equation}
with each $\{v_k,v_{k+1} \} \in E(G)$.
A walk $\Gamma$ of the form \eqref{ecw} is called {\it even} (resp. {\it odd}) if $q$ is even  (resp. odd).
A walk $\Gamma$ of the form \eqref{ecw} is called 
{\it closed} if $v_{q+1} = v_1$.
Given an even closed walk $    \Gamma =(e_{i_1}, \dots, e_{i_{2q}})$ of $G$, let
   $$
    f_\Gamma := \prod_{k=1}^q x_{i_{2k-1}} - \prod_{k=1}^q x_{i_{2k}}
    .$$
It is easy to see that $f_\Gamma$ belongs to $I_{\epoly}$.
The most important fact is the following.

\begin{prop}[{\cite[Proposition 3.1]{Vecw}}]
    The toric ideal $I_{\epoly}$ is generated by 
    $$\{ f_\Gamma : \Gamma \mbox{ is an even closed walk of }G \}.$$
\end{prop}

More precisely, we have the following (\cite[Lemmas 5.10 and 5.11]{HHO}).

\begin{prop}\label{edge gene}
    The toric ideal $I_{\epoly}$ is generated by all $f_\Gamma$ such that $\Gamma$ is one of the following{\rm :}
    \begin{itemize}
        \item[(i)]
        $\Gamma$ is an even cycle of $G${\rm ;}
        \item[(ii)]
        $\Gamma=(C_1,C_2)$ where $C_1$ and $C_2$ are odd cycles of $G$ having exactly one common vertex{\rm ;}
        \item[(iii)]
        $\Gamma=(C_1,\Gamma_1,C_2,\Gamma_2)$ where $C_1$ and $C_2$ are odd cycles of $G$ without common vertices and where $\Gamma_1$ and $\Gamma_2$ are walks of $G$ such that $\Gamma_1$ combines $i \in V(C_1)$ with $j \in V(C_2)$
        and $\Gamma_2$ combines $j \in V(C_2)$ with $i \in V(C_1)$. 
    \end{itemize}

\end{prop}

On the other hand, a graph-theoretical characterization of indispensable binomials in $I_{\epoly}$ is given in \cite[Theorem 4.14]{RTT}.
Since the characterization is complicated, we introduce some partial results.
A {\em chord} of a walk $\Gamma$ in \eqref{ecw} is an edge $e \in E(G) \setminus E(\Gamma)$ of the form $e=\{v_i,v_j\}$ for some $1 \le i < j \le q+1$.
A {\it cycle} is a closed walk 
\begin{equation} \label{evencycleexample}    
    C =(\{v_1,v_2\}, \{v_2,v_3\}, \dots,\{v_q,v_{1}\}),
    \end{equation}
where $v_i \ne v_j$ for all $1 \le i < j \le q$.
When $C$ is even, an {\em even chord} (resp.~{\em odd chord}) of $C$ is a chord $e=\{v_i,v_j\}$
with $1 \le i < j \le q$ such that $j-i$ is odd (resp.~even).
Let $e=\{v_i,v_j\}$ and $e'=\{v_k,v_\ell\}$ be odd chords of an even cycle $C$ in \eqref{evencycleexample}
with $1 \le i < j \le q$ and $1 \le k < \ell \le q$.
We say that $e$ and $e'$ {\em cross} in $C$ if the following conditions are satisfied.
\begin{itemize}
    \item 
    Either $i < k < j < \ell$ or $k < i < \ell < j$; 
    \item 
    Either $\{v_i,v_k\}$, $\{v_j,v_\ell\}$ are edges of $C$
or $\{v_i,v_\ell\}$, $\{v_j,v_k\}$ are edges of $C$.
\end{itemize}
In addition, we say that $e$ and $e'$ {\em cross effectively} in $C$ if the following conditions are satisfied.
\begin{itemize}
    \item 
    Either $i < k < j < \ell$ or $k < i < \ell < j$; 
    \item 
    $k-i$ is odd (and hence all of $k-j$, $\ell-i$, $\ell -j$ are odd).
\end{itemize}
If $C_1$ and $C_2$ are cycles of $G$ with no common vertices, then a {\em bridge} between $C_1$ and $C_2$ is an edge $\{i,j\}$ of $G$ with $i \in V(C_1)$ and $j \in V(C_2)$.

\begin{prop} \label{ep:indispensable}
Let $G$ be a graph and let $\Gamma$ be an even closed walk of $G$.
\begin{itemize}
    \item[(a)] 
    If $\Gamma$ is an even cycle, then $f_\Gamma$ is indispensable
if and only if {\rm (i)} $\Gamma$ has no even chord and {\rm (ii)} $\Gamma$ has no odd chords $e$ and $e'$ which
cross effectively in $\Gamma$ {\rm (}\cite[Theorem 2.3]{OHindispensable}{\rm )}.
\item[(b)]
If $\Gamma$ has no chords and consists of either {\rm (i)} two odd cycles having exactly one common vertex or
{\rm (ii)} two odd cycles with no common vertices and joined by a path, then $f_\Gamma$ is fundamental, and hence indispensable
 {\rm (}\cite[Theorem 2.2]{RTT}{\rm )}.
\end{itemize}
\end{prop}

In addition, a graph-theoretical characterization of quadratic generation of the toric ideal $I_{\epoly}$ is known \cite{OHquad}.

\begin{definition}
We say that a graph $G$ satisfies (Q) if $G$ satisfies the following conditions:
\begin{itemize}
    \item[(i)]
    If $C$ is an even cycle of $G$ of length $\ge 6$, then either 
    $C$ has an even chord or $C$ has three odd chords $e$, $e'$, and $e''$ such that $e $ and $e'$ cross in $C$;
\item[(ii)]
If $C_1$ and $C_2$ are induced odd cycles having exactly one common
 vertex, then there exists an edge $\{i,j\} \notin  E(C_1) \cup E(C_2)$
 with $i \in V(C_1)$ and $j \in V(C_2)$;
\item[(iii)]
If $C_1$ and $C_2$ are induced odd cycles in the same connected component of $G$ having no common vertices,
then there exist at least two bridges between $C_1$ and $C_2$.
  
\end{itemize}
\end{definition}

\begin{prop}[{\cite[Theorem 1.2]{OHquad}}] \label{edgepolyquad}
Let $G$ be a graph. Then the toric ideal $I_{\epoly}$ is quadratic if and only if $G$ satisfies (Q).
\end{prop}

Recall that, given a graph $G$, each vertex of the edge polytope of $G$ arises from an edge of $G$.
Edges of the edge polytope $\epoly$ of $G$ are characterized as follows.
    
\begin{prop}[{\cite[Lemma 1.4]{OHsimple}}] \label{edge of edge}
     Let $e = \{i,j\}$ and $e'=\{k,\ell\}$ $(e \ne e')$ be edges of a graph $G$.
     Then, the convex hull of $\{\eb_i+\eb_j, \eb_k+\eb_\ell\}$ is an edge of the edge polytope $\epoly$
     if and only if one of the following
 conditions is satisfied.
 \begin{itemize}
     \item[(i)]
$e$ and $e'$ have a common vertex{\rm ;}
     \item[(ii)]
      $e$ and $e'$ have no common vertices, and the induced subgraph of $G$ on the vertex set $\{i,j,k,\ell\}$ has no cycles of length $4$.
 \end{itemize}
\end{prop}

\begin{lemma}\label{edgepolyE}
    Let $G$ be a graph. Then the edge polytope $\epoly$ of $G$ satisfies (E).
\end{lemma}

\begin{proof}
 Let $e = \{i,j\}$ and $e'=\{k,\ell\}$ $(e \ne e')$ be edges of a graph $G$.
Suppose that the convex hull of  $\{\eb_i+\eb_j, \eb_k+\eb_\ell\}$ is not an edge of  $\epoly$ of $G$.
From Proposition~\ref{edge of edge}, 
    $e$ and $e'$ have no common vertices, and the induced subgraph of $G$ on the vertex set $\{i,j,k,\ell\}$ has a cycle of length $4$.
Then either $\{\{i,\ell\}, \{j,k\}\} \subset E(G)$ or $\{\{i,k\}, \{j,\ell\}\} \subset E(G)$.
If $\{\{i,\ell\}, \{j,k\}\} \subset E(G)$, then $(\eb_i+\eb_j) + (\eb_k+\eb_\ell) = (\eb_i+\eb_\ell) + (\eb_j+\eb_k)$.
Similarly, if $\{\{i,k\}, \{j,\ell\}\}  \subset E(G)$, then $(\eb_i+\eb_j) + (\eb_k+\eb_\ell) = (\eb_i+\eb_k) + (\eb_j+\eb_\ell)$.
\end{proof}

We now show that three conditions in \eqref{three_conditions} are equivalent for edge polytopes.

\begin{thm}\label{edge_equivalent}
    Let $G$ be a graph.
    Then the following conditions are equivalent.
    \begin{itemize}
        \item[{\rm (i)}] 
        $G$ satisfies (Q){\rm ;}
    
        \item[{\rm (ii)}] 
        $I_{\epoly}$ is quadratic{\rm ;}

        \item[{\rm (iii)}] 
        $\epoly$ satisfies the clique-face property $(*)${\rm ;}
        \item[{\rm (iv)}] 
        $I_{\epoly}$ has no indispensable monomials of degree $\ge 3$.

    \end{itemize}
\end{thm}

\begin{proof}
From Corollary \ref{zero_one}, Theorem \ref{indispensable}, Proposition \ref{edgepolyquad}, and Lemma \ref{edgepolyE}, we have
    ``(i) $\Leftrightarrow$ (ii) $\Rightarrow $ (iii) $\Rightarrow $ (iv)".
    Hence it is enough to show that ``(iv) $\Rightarrow $ (i)".
Suppose that 
        $I_{\epoly}$ has no indispensable monomial of degree $\ge 3$.
        In particular, $I_{\epoly}$ has no indispensable binomial of degree $\ge 3$.
        From Proposition \ref{ep:indispensable} (b), 
        we may assume that
        \begin{itemize}
            \item
            there do not exist induced odd cycles \(C_1\) and \(C_2\) having exactly one common vertex such that there is no edge
\(\{i,j\} \notin E(C_1)\cup E(C_2)\) with
\(i\in V(C_1)\) and \(j\in V(C_2)\);

 \item 
 there do not exist induced odd cycles \(C_1\) and \(C_2\) having no common vertices such that there exists exactly one bridge between \(C_1\) and \(C_2\).
        \end{itemize}

 Suppose that induced odd cycles $C_1$ and $C_2$ 
 in the same connected component of $G$ have no common
 vertices, and there exists no bridge between $C_1$ and $C_2$.
Since $C_1$ and $C_2$ belong to the same connected component of $G$, there exists a path $\Gamma = (e_1,\dots,e_s)$ of $G$
from $i \in V(C_1)$ to $j \in V(C_2)$ with $s \ge 2$.
We may assume that $s$ is minimal among such $C_1$, $C_2$, and $\Gamma$.
Since $C_1$ and $C_2$ are induced odd cycles without bridges,
$C_1 \cup C_2$ is an induced subgraph of $G$.
In addition, from the minimality of $s$, $\Gamma$ is an induced path of $G$.
If $C_1 \cup \Gamma \cup C_2$ is an induced subgraph of $G$, then, 
from Proposition \ref{ep:indispensable} (b), $f_{\Gamma'}$ where $\Gamma' =(C_1, \Gamma ,C_2)$
 is indispensable, a contradiction. 
Thus $C_1 \cup \Gamma \cup C_2$ is not an induced subgraph.
We may assume that there exists an edge $e=\{k,\ell\}$ where $k \in V(C_1) \setminus \{i\}$, $\ell \in \Gamma \setminus \{i\}$.
Then $e_1 = \{i,\ell\}$ from the minimality of $s$.
It then follows that there exists an odd cycle $C_1'$ in $C_1 \cup \{e_1, e\}$
such that $\{e_1, e\} \subset C_1'$.
Taking an induced odd cycle in \(C_1\cup\{e_1,e\}\) containing \(\ell\),
we may assume that \(C'_1\) is induced.
Then 
$C_1'$ and $C_2$ have no common vertices,
and are joined by the path $\Gamma \setminus \{e_1\}$ of length $s-1$.
If \(C'_1\) and \(C_2\) had no bridge, then the path
\(\Gamma\setminus\{e_1\}\) would contradict the minimality of \(s\).
Hence 
$C_1'$ and $C_2$ have at least two bridges.
In particular, $s=2$ and two bridges are $e_2$
and $e'=\{\ell, m\}$ with $m \in V(C_2) \setminus \{j\}$.
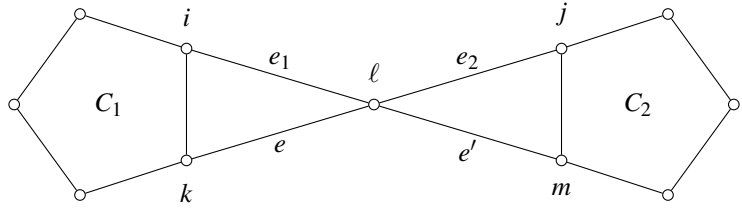
\begin{center}  
\begin{tikzpicture}[scale=0.7,transform shape,
    vertex/.style={circle,draw,minimum size=2mm,inner sep=0pt},
    every node/.style={font=\large}
]

\def\r{1.8}

\node[vertex] (L)  at ({-5+\r*cos(180)},{\r*sin(180)}) {};
\node[vertex] (Lb) at ({-5+\r*cos(252)},{\r*sin(252)}) {};
\node[vertex] (n)  at ({-5+\r*cos(324)},{\r*sin(324)}) {};
\node[vertex] (l)  at ({-5+\r*cos(36)},{\r*sin(36)}) {};
\node[vertex] (Lt) at ({-5+\r*cos(108)},{\r*sin(108)}) {};

\draw (L)--(Lt);
\draw (Lt)--(l);
\draw (l)--(n);
\draw (n)--(Lb);
\draw (Lb)--(L);
\node at (-5,0) {$C_1$};

\node[vertex] (c) at (0,0) {};

\node[vertex] (d)  at ({5+\r*cos(144)},{\r*sin(144)}) {};
\node[vertex] (Rt) at ({5+\r*cos(72)},{\r*sin(72)}) {};
\node[vertex] (R)  at ({5+\r*cos(0)},{\r*sin(0)}) {};
\node[vertex] (Rb) at ({5+\r*cos(-72)},{\r*sin(-72)}) {};
\node[vertex] (m)  at ({5+\r*cos(-144)},{\r*sin(-144)}) {};

\draw (d)--(Rt)--(R)--(Rb)--(m)--cycle;
\draw (d)--(m);

\node at (5,0) {$C_2$};

\draw (l)-- node[above] {$e_1$} (c);
\draw (n)-- node[below] {$e$} (c);

\draw (c)-- node[above] {$e_2$} (d);
\draw (c)-- node[below] {$e'$} (m);

\node[above=8pt] at (l) {$i$};
\node[below=8pt] at (n) {$k$};

\node[above=8pt] at (c) {$\ell$};

\node[above=8pt] at (d) {$j$};
\node[below=8pt] at (m) {$m$};

\end{tikzpicture}
\end{center}

\noindent
It then follows that there exists an odd cycle $C_2'$ in $C_2 \cup \{e_2, e'\}$
such that $\{e_2, e'\} \subset C_2'$.
Taking an induced odd cycle in \(C_2\cup\{e_2,e'\}\) containing \(\ell\),
we may assume that \(C'_2\) is induced.
Then $C_1'$ and $C_2'$ have exactly one common
 vertex $\ell$,  and there exists no edge 
 \(\{u,v\}\notin E(C'_1)\cup E(C'_2)\) with
\(u\in V(C'_1)\) and \(v\in V(C'_2)\), a contradiction.

Suppose that there exists an even cycle $C$ of $G$ of length $t \ge 6$ such that
    $C$ has no even chord and $C$ has no three odd chords $e$, $e'$, and $e''$ such that $e $ and $e'$ cross in $C$.
    We may assume that $t$ is minimal among such even cycles.
    Since $f_C$ is not indispensable, from Proposition \ref{ep:indispensable} (a), $C$ has odd chords $e$ and $e'$ such that $e $ and $e'$ cross effectively in $C$.
    Then $C \cup \{e,e'\}$ has two even cycles $C_1$ and $C_2$ with $e, e' \in C_i$ for each $i=1,2$.
    We may assume that the length of $C_1$ is minimal among such $e$, $e'$, $C_1$ and $C_2$.  
    Then $f_{C_1}$ is quadratic if and only if $e$ and $e'$ cross in $C$.

\bigskip

\noindent
{\bf Case 1} ($H=C \cup \{e,e'\}$ is an induced subgraph of $G$){\bf .}
Since \(H\) is an induced subgraph of \(G\), \(\mathcal{P}_H\) is a face of \(\mathcal{P}_G\).
Hence an indispensable monomial of \(I_{\mathcal{P}_H}\) is also indispensable for
\(I_{\mathcal{P}_G}\) by Lemma \ref{cp} (iii).
    If $e$ and $e'$ do not cross in $C$, then both $f_{C_1}$ and $f_{C_2}$ are indispensable binomials of degree $\ge 3$, a contradiction.
    Hence $e$ and $e'$ cross in $C$.

\medskip \noindent \emph{Claim.} After relabeling \(C_1\) and \(C_2\) if necessary, the minimal systems of binomial generators of \(I_{\mathcal P_H}\) are, up to nonzero scalar multiples, \[ \{f_{C_1},f_C\} \qquad\text{and}\qquad \{f_{C_1},f_{C_2}\}. \] Moreover, \(f_C\) and \(f_{C_2}\) have a common monomial of degree at least \(3\). 

\medskip

Indeed, \(H\) is a subdivision of \(K_4\) in which every replacement path has odd length. Its only primitive walks giving nonzero binomials are the three even cycles \(C,C_1,C_2\); see \cite[Theorem~3.1]{RTT}. Hence the only possible minimal binomials of \(I_{\mathcal P_H}\) are \(f_C,f_{C_1},f_{C_2}\). The criteria in \cite[Theorems~4.13 and 4.14 and Proposition~4.10]{RTT} show that \(f_{C_1}\) is indispensable and that \(f_C\) and \(f_{C_2}\) are minimal. With suitable choices of signs, we may write \[ f_C=u-v,\qquad f_{C_2}=v-w,\qquad m f_{C_1}=u-w \] for some monomials \(u,v,w\) and \(m\). Thus \[ f_C+f_{C_2}=m f_{C_1}. \] It follows that the two displayed pairs generate \(I_{\mathcal P_H}\). Since \(f_{C_1}\) is indispensable and there are no other possible minimal binomials, these are precisely the minimal systems of binomial generators. The common monomial \(v\) has degree at least \(3\). This proves the claim. 

\medskip

Note that the monomial \(v\) occurs in every minimal system of binomial generators of \(I_{\mathcal P_H}\). Since every binomial generating set contains a minimal generating subset, \(v\) is an indispensable monomial of \(I_{\mathcal P_H}\). By Lemma~\ref{cp} (iii), it is also an indispensable monomial of \(I_{\mathcal{P}_G}\), a contradiction.

\bigskip

\noindent
{\bf Case 2} ($C \cup \{e,e'\}$ is not an induced subgraph of $G$){\bf .}    
    Then $C$ has an odd chord $e''$ $(\ne e, e')$.  
    If $e$ and $e'$ cross in $C$, then $e,e',e''$ are three odd chords satisfying condition (Q) (i), a contradiction. Hence $e$ and $e'$ do not cross in $C$, and therefore
    the length of $C_1$ is at least $6$.
    If $\widetilde{e}$ is a chord of $C_1$, then $\widetilde{e}$ is an odd chord of $C_1$ if and only if $\widetilde{e}$ is an even chord of $C$.
    Hence $C_1$ has no odd chords, and has at least one even chord $e_1$
    by the hypothesis of minimality of the length of $C$.
    Then either ``$e$ and $e_1$" or ``$e'$ and $e_1$" cross effectively in $C$,
    and contradicts minimality of the length of $C_1$.

\bigskip

    Therefore, $G$ satisfies (Q), as desired.
\end{proof}

\begin{Remark} \label{not equiv}
    Each of the conditions in Theorem \ref{edge_equivalent} is \emph{not} equivalent to
    \begin{itemize}
        \item[(v)]
                $I_{\epoly}$ has no indispensable {\it binomials} of degree $\ge 3$.
    \end{itemize}
    In fact, if $G$ is a graph on the vertex set $\{1,2,3,4,5,6\}$ whose edge set is
    $$E(G)= \{ \{1,2\},\{ 2, 3\}, \{ 3,4\}, \{ 4, 5\}, \{ 5, 6\},\{ 1,6 \}, \{ 1, 3\}, \{ 2,4 \}   \},$$
    then
    \[
    I_{\epoly} = 
    \langle x_1x_3x_5- x_2x_4x_6, x_1 x_3 -x_7 x_8\rangle
    =\langle x_5x_7x_8- x_2x_4x_6, x_1 x_3 -x_7 x_8\rangle  
    .\]
    The monomial $x_2x_4x_6$ is an indispensable monomial of degree $3$.
    On the other hand, there exist no indispensable binomials of degree $\ge 3$.
\end{Remark}
    
\subsection{Cut polytopes}
Let $G$ be a connected graph on the vertex set $V=\{1,2,\dots,n\}$
and the edge set $E$. 
Given $S\subset V$, the cut semimetric on $G$ induced by $S$ is the $(0,1)$-vector $\delta_G(S)=(d_{ij} : \{i,j\}\in E) \in \RR^{E}$, where
$$
d_{ij}
:=
\left\{
\begin{array}{cl}
1     & \mbox{if } |S\cap \{i,j\}|=1, \\
0     & \mbox{otherwise } 
\end{array}
\right.
$$
for each $\{i,j\}\in E$.
In particular, $\delta_G(\emptyset) = {\bf 0}$.
Note that $\delta_G(S) = \delta_G(V \setminus S)$.
The vector $\delta_G(S)$ is the incidence vector
of a {\it cut} $\{e \in E : |S\cap e|=1 \}$ associated with $S$.
The \textit{cut polytope} ${\rm Cut}^\Box(G)$ of $G$ is the convex hull of $\{\delta_G(S) : S \subset V\} \subset  \ZZ^E$.
In general, ${\rm Cut}^\Box(G)$ is a full-dimensional $(0,1)$-polytope with $2^{n-1}$ vertices.
Let $I\ \Delta \ J$ denote the symmetric difference of $I$ and $J$, i.e., 
$I\ \Delta\ J := (I \setminus J ) \cup  (J\setminus I)$.
Note that if $I$ and $J$ are cuts, then so is $I \ \Delta \ J$.
Edges of ${\rm Cut}^\Box(G)$ are characterized as follows.

\begin{prop}[{\cite[Theorem 4.1]{BM}}]
\label{mmprop}
    Let $G=(V,E)$ be a connected graph and let 
    $I$ and $J$ be cuts associated with $S, S' \subset V$, respectively.
    Then $\delta(S)$ and $\delta(S')$ are adjacent in ${\rm sk}({\rm Cut}^\Box(G))$
    if and only if 
   the graph $H=(V, E \setminus (I \ \Delta \ J))$
has exactly two connected components.
\end{prop}

In \cite[Proof of Theorem 4.1]{BM},
it was shown that,
if the graph $H=(V, E \setminus (I \ \Delta \ J))$ has connected components 
$(V_1,E_1), \dots, (V_k,E_k)$ with $k \ge 3$
and $K$ is the cut associated with $V_1$,
then 
$\delta(S) + \delta(S') = \delta(T) + \delta(T') $
where $T$ and $T'$ correspond to cuts
$I \ \Delta \ K$ and $J \ \Delta \ K$.
Thus we have the following.

\begin{lemma}
    Let $G$ be a connected graph. Then the cut polytope ${\rm Cut}^\Box(G)$ of $G$ satisfies (E).
\end{lemma}

    It is known {\cite[Section 4.2]{Z}} that 
    the cut polytope
    ${\rm Cut}^\Box(K_n)$ of the complete graph $K_n$ is 3-neighborly but not 4-neighborly 
    for all $n\ge 4$.
    In particular, ${\rm Cut}^\Box(K_n)$ does not satisfy ($*$) if $n \ge 4$.
    On the other hand, it is known \cite{Eng} that $I_{{\rm Cut}^\Box(G)}$ is quadratic if and only if $G$ has no $K_4$-minors.

\begin{example}\label{k4}

    Let $G$ be the complete graph $K_4$.
    Then ${\rm Cut}^\Box(G)$ is the convex hull of $\{\ab_1, \dots, \ab_8\}$ where
    $\ab_1 = \delta_G(\emptyset)$, $\ab_2 = \delta_G(\{1,2\})$, $\ab_3 = \delta_G(\{1,3\})$, $\ab_4 = \delta_G(\{1,4\})$, $\ab_5 = \delta_G(\{1\})$, $\ab_6 = \delta_G(\{2\})$, $\ab_7 = \delta_G(\{3\})$, $\ab_8 = \delta_G(\{4\})$.
    Then
    \[
    (\ab_1, \dots, \ab_8) = 
    \left(
    \begin{array}{cccccccc}
    0 & 0 & 1 & 1 & 1 & 1 & 0 & 0\\
    0 & 1 & 0 & 1 & 1 & 0 & 1 & 0\\
    0 & 1 & 1 & 0 & 1 & 0 & 0 & 1\\
    0 & 1 & 1 & 0 & 0 & 1 & 1 & 0\\
    0 & 1 & 0 & 1 & 0 & 1 & 0 & 1\\
    0 & 0 & 1 & 1 & 0 & 0 & 1 & 1
    \end{array}
    \right),
    \]
    where rows of the matrix are indexed by the edges $
    \{1,2\},\{1,3\},\{1,4\},\{2,3\},\{2,4\},\{3,4\}$.
    As stated in \cite[Example 1.1]{SS}, $I_{{\rm Cut}^\Box(G)}$ is generated by $f=x_1x_2x_3x_4 - x_5x_6x_7x_8$, and hence $f$ is indispensable.
\end{example}

We now show that three conditions in \eqref{three_conditions} are equivalent for cut polytopes.

\begin{thm}
    Let $G$ be a connected graph.
    Then the following conditions are equivalent.
    \begin{itemize}
        \item[{\rm (i)}] 
        $G$ has no $K_4$-minors{\rm ;}
    
        \item[{\rm (ii)}] 
        $I_{{\rm Cut}^\Box(G)}$ is quadratic{\rm ;}

        \item[{\rm (iii)}] 
        ${\rm Cut}^\Box(G)$ satisfies the clique-face property $(*)${\rm ;}

        \item[{\rm (iv)}] 
        $I_{{\rm Cut}^\Box(G)}$ has no indispensable monomials of degree $\ge 3${\rm ;}

        \item[{\rm (v)}] 
        $I_{{\rm Cut}^\Box(G)}$ has no indispensable binomials of degree $\ge 3$.
        
    \end{itemize}
\end{thm}

\begin{proof}
Recall that ``(i) $\Leftrightarrow$ (ii)" was shown in \cite{Eng}.
    Since ${\rm Cut}^\Box(G)$ satisfies (E), we have
    ``(ii) $\Rightarrow $ (iii) $\Rightarrow $ (iv) $\Rightarrow $ (v)".
Hence 
it is enough to show that ``(v) $\Rightarrow $ (i)".
Suppose that $G$ has $K_4$ as a minor.
Since $G$ is connected and $K_4$ is complete, $K_4$ can be obtained from $G$ by a sequence of edge contractions. 
Since edge contraction corresponds to taking a face of a cut polytope, ${\rm Cut}^\Box(K_4)$ is a face of ${\rm Cut}^\Box(G)$.
As explained in Example \ref{k4},  
the toric ideal of ${\rm Cut}^\Box(K_4)$ has an indispensable binomial $f$ of degree $4$.
Since
${\rm Cut}^\Box(K_4)$ is a face of ${\rm Cut}^\Box(G)$, $f$ is an indispensable binomial in $I_{{\rm Cut}^\Box(G)}$ by Lemma \ref{cp} (iv).
\end{proof}

\section{Classes of polytopes whose toric ideals are expected to be quadratic}

In this section, we study classes of lattice polytopes whose toric ideals are conjectured to be quadratic.

\subsection{Simple polytopes}

One of the most important conjectures on toric ideals is the B{\o}gvad Conjecture.

\begin{conjecture}[the B{\o}gvad Conjecture]
The toric ideal of any smooth polytope is quadratic.
\end{conjecture}

Since any smooth polytope is simple, we study the clique-face property ($*$) for 
simple polytopes (that is not necessarily a lattice polytope).

\begin{prop}[{\cite[Theorem 3.6.6]{PGbook}}]
Let $P$ be a simple polytope.
    Then every induced cycle of length $\le 5$
    in ${\rm sk}(P)$ corresponds to a face of $P$.
\end{prop}

In particular, any triangle of ${\rm sk}(P)$ corresponds to a face of a simple polytope $P$.   
We extend it to general cliques of ${\rm sk}(P)$.
We need the following fact.

\begin{prop}[{\cite[Theorem 2.8.6]{PGbook}}]
\label{ItsAFace}
Let $P$ be a simple $d$-polytope and let $0 \le k \le d-1$.
Suppose that $v$ is a vertex of $P$ and ${\rm conv}(v,v_1),\ldots,{\rm conv}(v,v_k)$ are $k$ edges of $P$ that are incident with $v$,
and $F$ is the smallest face of $P$ containing these edges.
    Then $F$ is a simple $k$-face of $P$.
\end{prop}

From this, we show that any simple polytope satisfies $(*)$.

\begin{thm}\label{simple_is_OK}
Let $P \subset \RR^d$ be a simple polytope.
Then \(P\) has the clique-face property \((*)\).
\end{thm}

\begin{proof}
Let \(C=\{v,v_1,\ldots,v_k\}\) be a \((k+1)\)-clique of \(\operatorname{sk}(P)\).
Since \(P\) is simple, we have \(k\le d\). If \(k=d\), set \(F=P\).
If \(k<d\), then by Proposition \ref{ItsAFace}, there exists a simple \(k\)-face \(F\)
of \(P\) containing the edges
${\rm conv}(v,v_1),\ldots,{\rm conv}(v,v_k)$.
Since \(F\) is a face of \(P\), every edge of \(P\) whose endpoints lie in
\(F\) is an edge of \(F\). Hence \(C\) is a \((k+1)\)-clique contained in
\({\rm sk}(F)\).

Moreover, since \(F\) is a simple \(k\)-polytope, every vertex of \(F\)
has degree \(k\) in \({\rm sk}(F)\). As \(C\) is a
\((k+1)\)-clique contained in \({\rm sk}(F)\), each vertex of
\(C\) is already adjacent to \(k\) vertices in \(C\). Thus no edge of
\({\rm sk}(F)\) joins a vertex of \(C\) to a vertex outside
\(C\). Since \({\rm sk}(F)\) is connected, \(C\) is the vertex
set of \(F\). Therefore \(C\) corresponds to the face \(F\).
\end{proof}

\subsection{Matroid polytopes}

First, based on the textbook \cite{Oxley}, we introduce the concept of matroids.
A \emph{matroid} $\mathcal{M}$ is an ordered pair $(E,\mathcal{I})$ where $E$ is finite and $\mathcal{I} \subseteq 2^E$ satisfies the following three conditions:
\begin{enumerate}
    \item $\emptyset \in \mathcal{I}$;
    \item if $I \in \mathcal{I}$ and $I' \subset I$, then $I' \in \mathcal{I}$;
    \item if $I_1, I_2 \in \mathcal{I}$ and $|I_1| < |I_2|$, then there exists an element $x \in I_2 \setminus I_1$ such that $I_1 \cup \{x \}\in \mathcal{I}$.
\end{enumerate}
The set $E$ is called the \emph{ground set} of $\mathcal{M}$ and the elements of $\mathcal{I}$ are called the \emph{independent sets} of $\mathcal{M}$.
The \emph{rank} of a matroid $\mathcal{M} = (E,\mathcal{I})$ is defined as 
\[
    {\rm rank} (\mathcal{M}) = \max_{I \in \mathcal{I}} |I|.
\]
An element of $\mathcal{I}$ of cardinality ${\rm rank} (\mathcal{M})$ is called a \emph{basis} of $\mathcal{M}$.
Let ${\mathcal B}$ be the set of all bases of ${\mathcal M}$.
Given a subset $ I \subset E=\{e_1,\dots,e_n\}$, 
we set $\rho(I) = \sum_{e_i \in I} {\bf e}_i \in \{0,1\}^n$.
Given a matroid $\mathcal{M}=(E,\mathcal{I})$ with $E=\{e_1,\ldots,e_n\}$,
the {\em independence polytope} ${\mathcal P}({\mathcal M})$ of ${\mathcal M}$ is the convex hull of
$$
\{ \rho (I) : I \in {\mathcal I}\},
$$
and the {\em basis polytope} ${\mathcal P}_{B} ({\mathcal M})$ of ${\mathcal M}$ is the convex hull of
$$
\{ \rho (B) : B \in {\mathcal B}\}.
$$

White's conjecture \cite{WC} states that the toric ideal of every matroid base polytope is generated by quadratic binomials arising from symmetric exchanges. Although this conjecture is now known to be false, the following weaker conjecture remains open.

\begin{conjecture}\label{Mconj}
The toric ideal of any matroid base polytope is quadratic.
\end{conjecture}

\begin{Remark} The conjecture that every matroid base polytope has a quadratic Gr\"obner basis was recently disproved independently in \cite{Tsuika, DFMR}. In \cite{DFMR}, cliques in the \(1\)-skeleton of the base polytope of the Fano matroid play an important role. After the first version of this paper was posted, Larson \cite{Las} constructed a counterexample to White's symmetric exchange conjecture. This counterexample does not disprove Conjecture~4.5, because it only shows that the symmetric exchange binomials do not generate the toric ideal. 
\end{Remark}

Edges of matroid polytopes are characterized as follows.

\begin{prop}[{\cite[Theorem 4.3]{HK}}]\label{edge_ind}
    Let ${\mathcal M}$ be a matroid.
    Then ${\rm conv}(\rho(I), \rho(I'))$ with $I,I' \in {\mathcal I}$ is an edge of ${\mathcal P}({\mathcal M})$ if and only if 
    one of the following conditions holds:
    \begin{itemize}
        \item[{\rm (i)}]
        $|I \ \Delta\  I'| = 1$, or
        \item[{\rm (ii)}]
        $|I \ \Delta\  I'| = 2$ and $I \cup I' \notin {\mathcal I}$.
    \end{itemize}
\end{prop}

\begin{thm}
    Let $\mathcal{M}$ be a matroid.
    Then ${\mathcal P}({\mathcal M})$ satisfies the clique-face property ($*$).
\end{thm}

\begin{proof}
It suffices to show that every maximal clique is the vertex set of a simplex face. Indeed, every clique is contained in a maximal clique, and every subset of the vertex set of a simplex face is again the vertex set of a face. 
Let $F =\{I_1, \dots, I_r\} \subset {\mathcal I}$ be a maximal clique of ${\rm sk}({\mathcal P}({\mathcal M}))$.
The assertion is clear for cliques of cardinality at most two, so we may assume that \(r>2\).
Let $k= \min_{I \in F} |I|$.
From Proposition \ref{edge_ind}, we have
$k \le \max_{I \in F} |I| \le k+1 .$

\bigskip

\noindent
{\bf Case 1} ($\max_{I \in F} |I|=k$){\bf .}
Then every $I_i , I_j \in F$ with $i \ne j$ satisfy condition (ii) in Proposition \ref{edge_ind}.
Let $I = I_1 \cap I_2$ and $I_1 \ \Delta\  I_2 = \{x,y\}$
where $x \in I_1$ and $y \in I_2$.
Then $I_1 \cup I_2 = I \sqcup  \{x,y\}$ does not belong to ${\mathcal I}$.

\bigskip

\noindent
{\bf Case 1.1} ($x \in I_3$){\bf .}
Suppose that $y \notin I_3$.
Since $x \notin I_2$ and $|I_2 \ \Delta\  I_3| =2$, we have
$I_2 \ \Delta\  I_3 = \{x,y\}$.
Hence $I_1 = I_3$, a contradiction.
Thus we have $y \in I_3$.
Hence $I_3 = \{x,y\} \sqcup ( I \setminus \{z\})$
for some $z \in I$.
By the same argument, it follows that each $I_i$ ($1 \le i \le r$) is of the form
$I_i = ( \{x,y\} \sqcup I) \setminus \{z\}$ for some $z \in \{x,y\} \sqcup I$.
Conversely, if $I' \in {\mathcal I}$ is of the form 
$I' = ( \{x,y\} \sqcup I) \setminus \{z\}$ for some $z \in \{x,y\} \sqcup I$, then $I'$ and $I_i \ (\ne I')$ with $1 \le i \le r$ satisfy condition (ii) in Proposition \ref{edge_ind}.
Since $F$ is maximal, we have $I' \in F$.
Let 
$${\bf w} = \sum_{e_i \in  I \sqcup  \{x,y\}} \eb_i - \sum_{e_i \notin  I \sqcup  \{x,y\}} \eb_i.$$
Then the inner product ${\bf w} \cdot \rho(I)$ with $I \in {\mathcal I}$
is less than or equal to $k$, 
and equal to $k$ if and only if $I$ belongs to $F$.
Thus $F$ corresponds to a face of ${\mathcal P}({\mathcal M})$.
Moreover, the vertices of this face are of the form \[ \rho(I\sqcup\{x,y\})- \rho(\{z\}), \] and hence they are affinely independent. Thus, \(F\) is the vertex set of a simplex face of 
${\mathcal P}({\mathcal M})$.

\bigskip

\noindent
{\bf Case 1.2} ($x \notin I_3$){\bf .}
Since $x \in I_1$ and  $|I_1 \ \Delta\  I_3| =2$, we have $I_3 = I \sqcup \{z\}$
for some $z \notin I$.
Since $I_2 \ne I_3$, $z \ne y$.
By the same argument, it follows that each $I_i$ ($1 \le i \le r$) is of the form
$I_i =I \sqcup \{z\}$ for some $z \notin I$.
Let $F' = \bigcup_{i=1}^r I_i$.
Note that $I \cup \{\alpha, \beta\}$ with $\alpha , \beta \in F' \setminus I$ does not belong to ${\mathcal I}$.
Let 
$${\bf w} = \mu \sum_{e_i \in  I} \eb_i +
\sum_{e_i \in  F' \setminus I} \eb_i - \mu \sum_{e_i \notin F'} \eb_i,$$
where $\mu$ is a sufficiently large number.
Then the inner product ${\bf w} \cdot \rho(I)$ with $I \in {\mathcal I}$
is less than or equal to $\mu (k-1) +1$, 
and equal to $\mu (k-1) +1$ if and only if $I$ belongs to $F$.
Thus $F$ corresponds to a face of ${\mathcal P}({\mathcal M})$.
Moreover, its vertices are of the form \[ \rho(I)+\rho(\{z\}), \] and hence they are affinely independent. Thus, \(F\) is the vertex set of a simplex face of \({\mathcal P}({\mathcal M})\).

\bigskip

\noindent
{\bf Case 2} ($\max_{I \in F} |I|=k+1$){\bf .}
Let $I_1 \in F$ with $|I_1| =k$, and $I_2 \in F$ with $|I_2| =k+1$.
Then $I_1 \cup \{x\} = I_2$ for some $x \in E$.
Suppose that $|I_3|=k$.
Then $I_3 \cup \{y\} = I_2$ for some $y \in I_1$.
Hence $I_1 \cup I_3 = I_2$ belongs to ${\mathcal I}$.
This contradicts condition (ii) in Proposition \ref{edge_ind}
for $I_1$ and $I_3$.
Thus $|I_i|=k+1$ and hence $I_1 \subsetneq I_i$ for each $2 \le i \le r$.
Let $F' = \bigcup_{i=1}^r I_i$.
Note that $I_1 \cup \{\alpha, \beta\}$ with $\alpha , \beta \in F' \setminus I_1$ does not belong to ${\mathcal I}$.
Let 
$${\bf w} =  \sum_{e_i \in  I_1} \eb_i  - \sum_{e_i \notin F'} \eb_i.$$
Then the inner product ${\bf w} \cdot \rho(I)$ with $I \in {\mathcal I}$
is less than or equal to $k$, 
and equal to $k$ if and only if $I$ belongs to $F$.
Thus $F$ corresponds to a face of ${\mathcal P}({\mathcal M})$.
Moreover, its vertices are \(\rho(I_1)\) and points of the form \[ \rho(I_1)+\rho(\{x\}), \] and hence they are affinely independent. Thus, \(F\) is the vertex set of a simplex face of \({\mathcal P}({\mathcal M})\).
\end{proof}

Since ${\mathcal P}_{B} ({\mathcal M})$ is a face of ${\mathcal P}({\mathcal M})$, we have the following.

\begin{cor}
        Let $\mathcal{M}$ be a matroid.
    Then ${\mathcal P}_{B} ({\mathcal M})$ satisfies the clique-face property ($*$).
\end{cor}

\subsection{Stable set polytopes}

Let $G$ be a simple graph on $[d]$. A subset $S \subset [d]$ is called a {\em stable set} (or an {\em independent set}) of $G$
if $\{i,j\} \notin E(G)$ for all $i,j \in S$ with $i \neq j$.
In particular, $\emptyset$ and $\{i\}$ with $i \in [d]$
are stable.
Let $S(G)$ denote
the set of all stable sets of $G$.
Then the \textit{stable set polytope} ${\rm STAB}(G)$ of $G$ is 
the convex hull of
\[
\{\rho(S) : S \in S(G) \}.
\]
It was shown in \cite{propertyE} that any stable set polytope of a graph satisfies (E).
Edges of stable set polytopes are characterized as follows.

\begin{prop}[{\cite[Theorem 6.2]{C}}]
        Let $G$ be a graph.
    Then ${\rm conv}(\rho(S), \rho(S'))$ with $S, S' \in S(G)$ is an edge of ${\rm STAB}(G)$ if and only if the subgraph $H$ of $G$ induced by $S \ \Delta \ S'$ 
is connected.
\end{prop}

On the other hand, we have the following conjecture.

\begin{conjecture}[\cite{ER, OST}] \label{sspconjecture}
Let $G$ be a perfect graph.
Then the following are equivalent:
\begin{itemize}
    \item[(i)] 
    $G$ is perfectly contractile;
    \item[(ii)]
$I_{{\rm STAB}(G)}$ is quadratic;
    \item[(iii)]
    $G$ contains no even antiholes and no odd prisms.
    \end{itemize}
    
\end{conjecture}

Both ``(i) $\Longrightarrow$ (iii)" and  ``(ii) $\Longrightarrow$ (iii)" hold in general \cite{LMR, OST}.
Conjecture \ref{sspconjecture} is true for several classes of graphs,
including dart-free graphs, even prism-free graphs, 
weakly chordal graphs,
Meyniel graphs, and perfectly orderable graphs;
see \cite{OTqtr} and the references therein.

\begin{prop}
    Let $G$ be a graph.
    If ${\rm STAB}(G)$ has the clique-face property \((*)\), then $G$ contains no even antiholes and no odd prisms.

\end{prop}

\begin{proof}
Suppose that $G$ has an even antihole or an odd prism.
Then proofs of \cite[Theorem 1.7]{OST} and \cite[Proposition 11]{MOS}
guarantee that $I_{{\rm STAB}(G)}$ has an indispensable binomial of degree $\ge 3$. 
From Theorem~\ref{indispensable},  there exists a clique of ${\rm sk}({\rm STAB}(G))$ 
that does not correspond to any face of ${\rm STAB}(G)$.    
\end{proof}

Hence we have
\[
\begin{array}{ccccc}
     &   &  &  & G \mbox{ is perfectly contractile}
\\     &   &  &  & \Downarrow
\\
I_{{\rm STAB}(G)} \mbox{ is quadratic }     &  \Rightarrow &
\mbox{the clique-face property } (*) &\Rightarrow &
\begin{array}{c} G \mbox{ contains no even antiholes}\\
 \mbox{and no odd prisms}.
 \end{array}
\end{array}
\]

\smallskip

\noindent
Thus these four conditions are expected to be equivalent for perfect graphs.


\begin{thebibliography}{99}

\bibitem{propertyE}
F. Aliniaeifard, C. Benedetti, N. Bergeron, S. X. Li, F. Saliola,
Stable set polytopes and their $1$-skeleta, preprint\\
{\tt arXiv:1804.00360}

\bibitem{Tsuika}
S. Backman, N. Cheung, M. Laso\'n, G. Liu, M. Micha\l{}ek,
The Gr\"obner Version of White's Conjecture is False,
{\tt arXiv:2606.13960}


\bibitem{BM}
F. Barahona and A. R. Mahjoub, 
On the cut polytope,
{\it Math. Program.} {\bf 36} (1986), 157--173.

\bibitem{BGT}
W. Bruns, J. Gubeladze and N. V. Trung,
Normal polytopes, triangulations, and Koszul algebras,
{\it J. Reine Angew. Math.} {\bf 485} (1997), 123--160.



\bibitem{C}
 V. Chv\'{a}tal, On certain polytopes associated with graphs, 
 {\it J. Combin. Theory Ser. B}  {\bf 18} (1975), 138--154.

\bibitem{DFMR}
J. A. De Loera, L. Ferroni, S. Morales and J. Rambau,
There are matroid toric ideals without quadratic Gr\"obner bases,
{\tt arXiv:2606.11014}.


\bibitem{Eng}
A. Engst\"{o}m, Cut ideals of $K_4$-minor free graphs are generated by quadrics, \textit{Michigan Math. J.} {\bf 60} (2011), 705--714.

\bibitem{ER}
H. Everett and B. A. Reed,
Problem session on parity problems,
{\em in} ``Perfect Graphs Workshop",
Princeton University, New Jersey, June 1993.


\bibitem{HK}
 D. Hausmann and B. Korte,
 Colouring criteria for adjacency on $0$-$1$ polyhedra, 
Math. Program. Stud. {\bf  8} (1978) 106--127. 

\bibitem{HHO}
J. Herzog, T. Hibi and H. Ohsugi, ``Binomial ideals"
Graduate Texts in Math. {\bf 279}, Springer, Cham, 2018.

\bibitem{Las}
M. Larson,
Counterexamples to two conjectures about matroids,
{\tt arXiv:2607.02208}


\bibitem{LMR}
C. Linhares Sales, F. Maffray and B. A. Reed,
On planar perfectly contractile graphs,
\textit{Graphs Combin.} {\bf 13} (1997), 167--187.

\bibitem{MOS}
K. Matsuda, H. Ohsugi, and K. Shibata,
Toric rings and ideals of stable set polytopes,
{\it Mathematics} 7:613, 2019.

\bibitem{Mor}
A. Mori,
Simplex faces of order and chain polytopes,
{\it Order} {\bf 42} (2025) 805--810.

\bibitem{cpureR1}
H. Ohsugi,
A geometric definition of combinatorial pure subrings and Gr\"obner bases of toric ideals of
positive roots,
{\it Comment. Math. Univ. St. Pauli} {\bf 56} (2007) 27--44.

\bibitem{cpure}
H. Ohsugi, J. Herzog and T. Hibi,
Combinatorial pure subrings,
\textit{Osaka J. Math.} {\bf 37} (2000), 745--757.


\bibitem{OHquad}
H. Ohsugi and T. Hibi, Toric ideals generated by quadratic binomials, \textit{J. Algebra}, {\bf 218} (1999),
 509--527.


 
\bibitem{OHindispensable}
H. Ohsugi and T. Hibi, Indispensable binomials of finite graphs, \textit{J. Algebra Appl.}, {\bf 4} (2005), 421--434.

\bibitem{OHsimple}
H. Ohsugi and T. Hibi, Simple polytopes arising from finite graphs, 
{\it in} ``Proceedings of the
2008 International Conference on Information Theory and Statistical Learning (ITSL)", 2008.




\bibitem{OST}
H. Ohsugi, K. Shibata and A. Tsuchiya,
Perfectly contractile graphs and quadratic toric rings,
{\em Bull. Lond. Math. Soc.} {\bf 55} (2023), 1264--1274.

\bibitem{OTqtr}
H. Ohsugi and A. Tsuchiya,
Kempe equivalence and quadratic toric rings,
{\it J. Pure Appl. Algebra} {\bf 230} (2026), 108257.



\bibitem{Oxley}
J. G. Oxley,
Matroid theory, volume 21 of Oxford Graduate Texts in Mathematics, Oxford University Press, Oxford, second edition, 2011.

\bibitem{RTT}
E. Reyes, C. Tatakis and A. Thoma,
Minimal generators of toric ideals of graphs,
{\it Adv. Appl. Math.} {\bf 48} (2012), 64--78.

\bibitem{SS}
B. Sturmfels and S. Sullivant, Toric geometry of cuts and splits,
\textit{Michigan Math. J.} {\bf 57} (2008), 689--709.

\bibitem{PGbook}
G.~P. Villavicencio,
``Polytopes and graphs",
Cambridge Stud. Adv. Math. {\bf 211},
Cambridge University Press, Cambridge, 2024. 

\bibitem{Vecw}
 R.H. Villarreal, Rees algebras of edge ideals, 
 {\it Comm. Algebra} {\bf 23} (1995), 3513--3524.

\bibitem{V}
R. H. Villarreal, ``Monomial Algebras", 2nd ed., Chapman \& Hall/CRC, Boca Raton, FL, 2015.

\bibitem{WC}
N. L. White, A unique exchange property for bases, 
\textit{Linear Algebra Appl.} {\bf 31} (1980),
81--91.

\bibitem{W}
H. Whitney,
Non-separable and planar graphs,
{\it Trans. Amer. Math. Soc.} {\bf 34} (1932), 339--362.


\bibitem{Z}
G. M. Ziegler, ``Lectures on $0/1$-Polytopes",
{\it In} Polytopes - Combinatorics and Computation,
(G. Kalai and G. Ziegler Eds)
DMV Seminar, vol 29. Birkh\"{a}user, Basel, 2000.
\end{thebibliography}
\end{document}